\documentclass[11pt]{article}

\usepackage{tikz}
\usepackage{subfigure}
\usepackage[english]{babel}

\usepackage[center]{caption2}
\usepackage{amsfonts,amssymb,amsmath,latexsym,amsthm}
\usepackage{multirow}
\usepackage[usenames,dvipsnames]{pstricks}
\usepackage{epsfig}
\usepackage{pst-grad} 
\usepackage{pst-plot} 
\usepackage[space]{grffile} 
\usepackage{etoolbox} 
\makeatletter 
\patchcmd\Gread@eps{\@inputcheck#1 }{\@inputcheck"#1"\relax}{}{}

\newtheorem{thm}{Theorem}
\newtheorem{myDef}[thm]{Definition}
\newtheorem{cor}[thm]{Corollary}
\newtheorem{lem}[thm]{Lemma}
\newtheorem{prop}[thm]{Proposition}

\newtheorem{quest}[thm]{Question}

\let\svthefootnote\thefootnote
\newcommand\blankfootnote[1]{%
	\let\thefootnote\relax\footnotetext{#1}%
	\let\thefootnote\svthefootnote%
}

\begin{document}

\title{\Large A property on monochromatic copies of graphs containing a triangle}
	
\date{}
	
\author{
Hao Chen~~~~~~~
Jie Ma
}

\blankfootnote{\noindent School of Mathematical Sciences, University of Science and Technology of China, Hefei, Anhui 230026, China.
Research supported by the National Key R and D Program of China 2020YFA0713100,
National Natural Science Foundation of China grant 12125106,
Innovation Program for Quantum Science and Technology grant 2021ZD0302902,
and Anhui Initiative in Quantum Information Technologies grant AHY150200.}
	
\maketitle
	
\begin{abstract}
A graph $H$ is called {\it common} and respectively, {\it strongly common} if the number of monochromatic copies of $H$ in a 2-edge-coloring $\phi$ of a large clique is asymptotically minimised by the random coloring with an equal proportion of each color and respectively, by the random coloring with the same proportion of each color as in $\phi$.
A well-known theorem of Jagger, {\v S}t'ov{\' i}{\v c}ek and Thomason states that every graph containing a $K_4$ is not common.
Here we prove an analogous result that every graph containing a $K_3$ and with at least four edges is not strongly common.
\end{abstract}

\section{Introduction}
A graph $H$ is said to be {\it common} if the number of monochromatic copies of $H$ in a 2-edge-coloring of a large clique is asymptotically minimised by the uniformly random 2-edge-coloring.
This can be viewed as a quantitative extension of Ramsey's Theorem and has been received extensive attention in the literature.
The classic formula of Goodman~\cite{G59} shows that $K_3$ is common.
Erd{\H{o}}s~\cite{E62} conjectured that all complete graphs are common and this conjecture was extended to all graphs by Burr and Rosta~\cite{BR80}.
Both conjectures fail to be true,
as Sidorenko~\cite{S89} proved that a triangle plus a pendant edge is not common and Thomason~\cite{T89} proved that $K_t$ is not common for any $t\geq 4$.
Jagger, {\v S}t'ov{\' i}{\v c}ek and Thomason~\cite{JST96} further showed that every graph containing a $K_4$ as a subgraph is not common.
We direct interested readers to \cite{BMN22,F08,GLLV,HHNRK12,JST96,KoL22,KVW,LN23,R16,S89,S96,T89} for many interesting results on this topic and to \cite{CHL19,CY11,HKKV22,KNNVW22} for generalizations.

A graph $H$ is {\it strongly common} if the number of monochromatic copies of $H$ in a 2-edge-coloring $\phi$ of a large clique is asymptotically minimised by the random coloring with the same proportion of each color as in $\phi$.
Extending the notion of common graphs, this natural definition was recently formalized by Behague, Morrison and Noel in \cite{BMN22}.
A famous conjecture of Sidorenko \cite{S93} asserts that for any bipartite graph $H$, the number of copies of $H$ in a graph $G$ is asymptotically minimised by the random graph with the same edge density as $G$.
It is evident to see that any bipartite graph satisfying Sidorenko's conjecture is strongly common (see \cite{CFS10,CKLL18,CL17,CL21,FW17,H10,JS21,LS11,LL11,S93,S931,Sz15} for advances on Sidorenko's conjecture).
For non-bipartite graphs, the authors of \cite{BMN22} proved that the triangle and the five-cycle are strongly common,
and they made a conjecture that all odd cycles are strongly common which was proved by Kim and Lee~\cite{KL22} very recently.
The authors of \cite{BMN22} also raised the problem of classifying strongly common graphs.
Towards an understanding of this problem, we prove the following result.
For a graph $H$, we say a graph $G$ {\it properly contains} $H$ if $H$ is a subgraph of $G$ and $e(H)<e(G)$.

\begin{thm}\label{thm:main}
Any graph properly containing a triangle is not strongly common.
\end{thm}

This can be viewed as an analogy of the result of Jagger, {\v S}t'ov{\' i}{\v c}ek and Thomason that every graph containing a $K_4$ is not common.
Behague, Morrison and Noel commented in \cite{BMN22} that they were unaware of any graph which is common but not strongly common.
This was addressed in one of the results in the recent paper of Lee and Noel \cite{LN23}, where they found many non-bipartite common graphs which are not strongly common (including the disjoint union of two triangles).
Let us mention a recent beautiful result of Grzesik, Lee, Lidick{\'y} and Volec~\cite{GLLV} which says that any {\it triangle-tree} is common.\footnote{A {\it triangle-tree} is either a triangle or it is obtained from a triangle-tree by identifying a single vertex or an edge of a new triangle with a vertex or an edge.}
Together with this result, we now see a rich family of such graphs in the following.

\begin{cor}
Any triangle-tree (except the triangle itself) is common but not strongly common.
\end{cor}


The rest of the paper is organized as follows.
In Section 2, we give out some preliminaries used later.
In Section 3, we build up some key lemmas on a special kernel called $U_p$.
In Section 4, we prove Theorem~\ref{thm:main}.
In Section 5, we conclude this paper with a question and some remarks.

\section{Preliminaries}
Let $G$ be a graph with the vertex set $V(G)$ and the edge set $E(G)$.
Denote the numbers of vertices and edges of $G$ by $v(G)$ and $e(G)$, respectively.
For a subset $A\subseteq V(G)$, we write $G[A]$ for the subgraph of $G$ induced by $A$.
We write $[k]$ for the set $\{1,2,...,k\}$ for $k\in \mathbf{Z}^+$.

In this paper, all integrations are taken with respect to the Lebesgue measure and when we use the notion of measure we always mean Lebesgue measure.
For a measurable set $S$, we let $\mu(S)$ be its Lebesgue measure.
A {\it kernel} is a measurable and bounded function $U:[0,1]^2\rightarrow \mathbb{R}$ such that $U(x,y)=U(y,x)$ for every $(x,y)\in [0,1]^2$.
\begin{myDef}
Let $H$ be a graph. For a kernel $U$, the homomorphism density from $H$ to $U$ is defined as
$$t_H(U)=\int_{[0,1]^{v(H)}}{\prod_{ij\in E(H)}U(x_i,x_j)\prod_{i=1}^{v(H)}dx_i}.$$
\end{myDef}	
Using this notation, we see (i.e., Lov{\'a}sz \cite{LL12}) that a graph $H$ is {\it strongly common}, if the inequality
\begin{equation}\label{equ:str-common}
t_H(W)+t_H(1-W)\geq t_{K_2}(W)^{e(H)}+t_{K_2}(1-W)^{e(H)}
\end{equation}
holds for every kernel $W:[0,1]^2\rightarrow [0,1]$.
In what follows, we transform the above {\it strongly commonality} of graphs to a formulae which is easier to deal with (i.e., see Kim and Lee \cite{KL22}).
Let $\mathcal{E}^{+}(H)$ be the set of all spanning subgraphs $F$ of $H$ with positive even number of edges.
For a kernel $W:[0,1]^2\rightarrow [0,1]$, we define $U:=2W-1$.
So $U$ is a kernel taking values in $[-1,1]$.
By standard multilinear expansion, we have
\begin{equation}
\begin{aligned}
t_H(W)+t_H(1-W)&=\frac{t_H(1+U)+t_H(1-U)}{2^{e(H)}}=2^{1-e(H)}\left(1+\sum_{F\in \mathcal{E}^+(H)}t_F(U)\right).
\end{aligned}
\end{equation}
Similarly, we can expand $t_{K_2}(W)^{e(H)}+t_{K_2}(1-W)^{e(H)}$ to obtain
\begin{equation}
\begin{aligned}
t_{K_2}(W)^{e(H)}+t_{K_2}(1-W)^{e(H)}=2^{1-e(H)}\left(1+\sum_{F\in \mathcal{E}^+(H)}t_{K_2}(U)^{e(F)}\right).
\end{aligned}
\end{equation}
Thus to answer if \eqref{equ:str-common} holds for $H$, we just need to determine the sign of the formula $$\sum_{F\in \mathcal{E}^+(H)}\left(t_F(U)-t_{K_2}(U)^{e(F)}\right).$$

We summarize this as the following property.

\begin{prop}\label{prop}
Let $H$ be a graph. If $\sum_{F\in \mathcal{E}^+(H)}\left(t_F(U)-t_{K_2}(U)^{e(F)}\right)\geq 0$ holds for every kernel $U$ taking values in $[-1,1]$, then $H$ is strongly common.
On the other hand, if there exists a kernel $U$ taking values in $[-1,1]$ such that $\sum_{F\in \mathcal{E}^+(H)}\left(t_F(U)-t_{K_2}(U)^{e(F)}\right)<0,$ then $H$ is not strongly common.
\end{prop}

\section{The kernel $U_p$}
In this section we consider the following kernel $U_p$ and several propositions on its homomorphism density $t_F(U_p)$ which play important roles in the proofs.
For $p\in [0,1]$, we define
\begin{equation}
U_p(x,y):=\left\{
\begin{aligned}
&2p-1 & \mbox{ if } (x,y)\in [0,1/2)\times[0,1/2)\ \mbox{or} \ [1/2,1]\times [1/2,1];\\
&-1 & \mbox{ if } (x,y)\in [0,1/2)\times [1/2,1]\ \mbox{or} \ [1/2,1]\times [0,1/2).
\end{aligned}
\right.
\end{equation}
The kernel $U_p$ can be viewed as a weighted complete graph whose vertex set is a disjoint union of two subsets $V_1$ and $V_2$,
where each of $V_1$ and $V_2$ induces a complete graph with edge-weight $2p-1$ for every edge, and all edges between $V_1$ and $V_2$ have edge-weight $-1$.

It is easy to see $t_{K_2}(U_p)=p-1$.
Throughout this section, we let $I_1=[0,1/2)$ and $I_2=[1/2,1]$.
For a string $S=s_1s_2...s_{v(H)}\in \{1,2\}^{v(H)}$,
we define $$B_S=\{(x_1,x_2,...,x_{v(H)}):x_k\in I_{s_k},k=1,2,...,v(H)\}$$ to be the corresponding domain in $[0,1]^{v(H)}$ determined by $S$.
There are $2^{v(H)}$ strings in total (say $S_1, S_2,...,S_{2^{v(H)}}$) and let $\mathcal{B}=\{B_{1},B_{2},...,B_{2^{v(H)}}\}$ be the collection of all domains in $[0,1]^{v(H)}$ corresponding to strings $S_i$ for $1\leq i\leq 2^{v(H)}$.
Note that for any $i,j\in \{1,2\}$, if $x, y$ are in the same interval $I_i$ and $x',y'$ are in the same interval $I_j$ as well,
then by the definition of $U_p$, we have $U_p(x,x')=U_p(y,y')$.
Consider any $B_k\in \mathcal{B}$.
Then $\mu(B_k)=1/2^{v(H)}$ and for any tuple $(x_1,x_2,...,x_{v(H)})\in B_k$,
the value $h_k:=\prod_{ij\in E(H)}U_p(x_i,x_j)$ only depends on $B_k$.
Thus, we have
\begin{equation}
t_H(U_p)=\int_{[0,1]^{v(H)}}\prod_{ij\in E(H)}U_p(x_i,x_j)\prod_{i=1}^{v(H)}dx_i=\frac{1}{2^{v(H)}}\sum_{k=1}^{2^{v(H)}}h_k.
\end{equation}
Next we introduce two variants of $t_H(U_p)$. For $a,b\in V(H)$ and $1\leq i\neq j\leq 2$, we define
$$f_{a,b,H}^i(p)=\int_{I_i\times I_i}\left(\int_{[0,1]^{v(H)-2}}\prod_{uv\in E(H)}U_p(x_u,x_v)\prod_{k\in V(H)\backslash\{a,b\}}dx_k\right)dx_adx_b$$ and
$$g_{a,b,H}^{ij}(p)=\int_{I_i\times I_j}\left(\int_{[0,1]^{v(H)-2}}\prod_{uv\in E(H)}U_p(x_u,x_v)\prod_{k\in V(H)\backslash\{a,b\}}dx_k\right)dx_adx_b.$$
It is clear that by symmetric, $f_{a,b,H}^1=f_{a,b,H}^2$ and $g_{a,b,H}^{12}=g_{a,b,H}^{21}$.
For brevity, we let $f_{a,b,H}:=f_{a,b,H}^1=f_{a,b,H}^2$ and $g_{a,b,H}:=g_{a,b,H}^{12}=g_{a,b,H}^{21}$.
Thus we have $$t_H(U_p)=f_{a,b,H}^1(p)+f_{a,b,H}^2(p)+g_{a,b,H}^{12}(p)+g_{a,b,H}^{21}(p)=2f_{a,b,H}(p)+2g_{a,b,H}(p).$$
Also by the definition, we have that $f_{a,b,H}(0)=(-1)^{e(H)}/4$.

The following lemma is key for the proof of our main result.
For a polynomial $f(p)$ on the variable $p$, we define $[j]f$ as the coefficient of $p^j$ in $f(p)$.

\begin{lem}\label{lem:key}
Let $H$ be a graph and $a,b$ be two vertices with $ab\notin E(H)$.
Write $f=f_{a,b,H}$ and $g=g_{a,b,H}$.
Then for any $j=0,1,...,e(H)$, we have $[j]f\cdot [j]g\geq 0$.
Moreover, $[0]f=[0]g$, $[1]f=[1]g$ and $|[2]f|\geq |[2]g|$.
\end{lem}

\begin{proof}
First note that $4\cdot [0]f=4f_{a,b,H}(0)=(-1)^{e(H)}=4g_{a,b,H}(0)=4\cdot [0]g,$ implying that $[0]f=[0]g$.
In this proof, we use $x_a$ and $x_b$ to express the first two variables in $[0,1]^{v(H)}$, which correspond to the vertices $a$ and $b$ respectively.

Let $\mathcal{B}_f=\{B_1,B_2,...,B_{2^{v(H)-2}}\}$ be the collection of domains in $[0,1]^{v(H)}$ such that each $B_k\in \mathcal{B}_f$ satisfies $x_a,x_b\in I_1$.
Let $\mathcal{B}_g=\{B_{2^{v(H)-2}+1},B_{2^{v(H)-2}+2},...,B_{2^{v(H)-1}}\}$ be the  collection of domains in $[0,1]^{v(H)}$ such that each $B_k\in \mathcal{B}_f$ satisfies $x_a\in I_1$ and $x_b\in I_2.$
Without loss of generality, we may assume the following. For any $i_s\in \{1,2\}$, we write $j_s\in \{1,2\}\backslash \{i_s\}$.
If $B_k\in \mathcal{B}_f$ has the form $B_k=I_1\times I_1\times \prod_{s=1}^{2^{v(H)-2}} I_{i_s}$,
then
\begin{itemize}
\item $B_{2^{v(H)-2}-k+1}\in \mathcal{B}_f$ satisfies $B_{2^{v(H)-2}-k+1}=I_1\times I_1\times \prod_{s=1}^{2^{v(H)-2}} I_{j_s},$
\item $B_{2^{v(H)-2}+k}\in \mathcal{B}_g$ satisfies $B_{2^{v(H)-2}+k}=I_1\times I_2\times \prod_{s=1}^{2^{v(H)-2}} I_{i_s},$ and
\item $B_{2^{v(H)-1}-k+1}\in \mathcal{B}_g$ satisfies $B_{2^{v(H)-1}-k+1}=I_1\times I_2\times \prod_{s=1}^{2^{v(H)-2}} I_{j_s}.$
\end{itemize}
Recall that $h_k:=\prod_{ij\in E(H)}U_p(x_i,x_j)$ for any $(x_a,x_b,x_3,x_4,...,x_{v(H)})\in B_k$.
Then we have
\begin{equation}\label{equ:f&g}
4f_{a,b,H}=\sum_{k=1}^{2^{v(H)-2}}\frac{1}{2^{v(H)-2}}h_{k} \mbox{ ~ and ~ } 4g_{a,b,H}=\sum_{k=2^{v(H)-2}+1}^{2^{v(H)-1}}\frac{1}{2^{v(H)-2}}h_{k}.
\end{equation}
Let $H^*$ be obtained from $H$ by deleting the vertices $a$ and $b$ and for a given $B_k$,
we define $$h^*_{k}=\prod_{ij\in E(H^*)}U_p(x_i,x_j) \mbox{~ for any ~} (x_a,x_b,x_3,x_4,...,x_{v(H)})\in B_k.$$
Then we may assume that $h^*_{k}=(2p-1)^{t_{k}}(-1)^{e(H^*)-t_{k}}$ for some non-negative integer $t_k$.
Let $A=|N(a)|, B=|N(b)|$, and $C=|N(a)\cap N(b)|.$
We further assume that there are $\alpha^+$ and $\alpha^-$ vertices in $N(a)/(N(a)\cap N(b))$ whose corresponding variables in $B_k$ belong to $I_1$ and $I_2$, respectively;
there are $\beta^+$ and $\beta^-$ vertices in $N(b)/(N(a)\cap N(b))$ whose corresponding variables in $B_k$ belong to $I_1$ and $I_2$, respectively;
and there are $c^+$ and $c^-$ vertices in $N(a)\cap N(b)$ whose corresponding variables in $B_k$ belong to $I_1$ and $I_2$, respectively.
By these definitions, we have that $A=\alpha^++\alpha^- +c^++c^-$, $B=\beta^++\beta^-+c^++c^-$, and $C=c^++c^-$.
Since $B_k\in \mathcal{B}_f$, we can rewrite
\begin{equation*}
\begin{aligned}
h_{k}&=h_k^*\cdot(2p-1)^{\alpha^++2c^++\beta^+}(-1)^{\alpha^-+2c^-+\beta^-}=(2p-1)^{t_k+\alpha^++2c^++\beta^+}(-1)^{e(H^*)-t_k+\alpha^-+2c^-+\beta^-}.
\end{aligned}
\end{equation*}
As $e(H)-e(H^*)=\alpha^++\alpha^-+\beta^++\beta^-+2c^++2c^-$,
the coefficient of $p^j$ in $h_k$ is
\begin{equation*}
\begin{aligned}
\mbox{$[j]h_k$}=& 2^j\binom{t_k+\alpha^++2c^++\beta^+}{j}(-1)^{(t_k+\alpha^++2c^++\beta^+-j)+(e(H^*)-t_k+\alpha^-+2c^-+\beta^-)}\\
=& 2^j\binom{t_k+\alpha^++2c^++\beta^+}{j}(-1)^{e(H)-j},
\end{aligned}
\end{equation*}

Next consider $B_{2^{v(H)-2}-k+1}\in \mathcal{B}_f$.
By symmetric, we have $h^*_{{2^{v(H)-2}-k+1}}=h^*_{k}$ and thus
\begin{equation*}
\begin{aligned}
h_{2^{v(H)-2}-k+1}=h^*_k\cdot (2p-1)^{\alpha^-+2c^-+\beta^-}(-1)^{\alpha^++2c^++\beta^+}=(2p-1)^{t_k+\alpha^-+2c^-+\beta^-}(-1)^{e(H^*)-t_k+\alpha^++2c^++\beta^+}.
\end{aligned}
\end{equation*}
Then the coefficient of $p^i$ in $h_{{2^{v(H)-2}-k+1}}$ is
\begin{equation*}
\begin{aligned}
\mbox{$[j]h_{{2^{v(H)-2}-k+1}}$}=2^j\binom{t_k+\alpha^-+2c^-+\beta^-}{j}(-1)^{e(H)-j}.
\end{aligned}
\end{equation*}
Consider $B_{2^{v(H)-2}+k}\in \mathcal{B}_g$. Again we have $h^*_{{2^{v(H)-2}+k}}=h^*_{k}$. So it follows that $h_{2^{v(H)-2}+k}$ equals
\begin{equation*}
\begin{aligned}
h^*_k\cdot (2p-1)^{\alpha^++c^++c^-+\beta^-}(-1)^{\alpha^-+c^-+c^++\beta^+}=(2p-1)^{t_k+\alpha^++c^++c^-+\beta^-}(-1)^{e(H^*)-t_k+\alpha^-+c^++c^-+\beta^+}
\end{aligned}
\end{equation*}
and the coefficient of $p^j$ in $h_{{2^{v(H)-2}+k}}$ is
\begin{equation*}
\begin{aligned}
\mbox{$[j]h_{2^{v(H)-2}+k}$}=2^j\binom{t_k+\alpha^++c^++c^-+\beta^-}{j}(-1)^{e(H)-j}.
\end{aligned}
\end{equation*}
Finally consider $B_{2^{v(H)-1}-k+1}\in \mathcal{B}_g.$
Since $h^*_{{2^{v(H)-1}-k+1}}=h^*_{k}$, we have $h_{2^{v(H)-1}-k+1}$ equals
\begin{equation*}
\begin{aligned}
h^*_{k}\cdot (2p-1)^{\alpha^-+c^-+c^++\beta^+}(-1)^{\alpha^++c^++c^-+\beta^-}=(2p-1)^{t_k+\alpha^-+c^-+c^++\beta^+}(-1)^{e(H^*)-t_k+\alpha^++c^++c^-+\beta^-}
\end{aligned}
\end{equation*}
and the coefficient of $p^j$ in $h_{{2^{v(H)-1}-k+1}}$ is
\begin{equation*}
\begin{aligned}
\mbox{$[j]h_{2^{v(H)-1}-k+1}$}=2^j\binom{t_k+\alpha^-+c^++c^-+\beta^+}{j}(-1)^{e(H)-j}.
\end{aligned}
\end{equation*}

Using \eqref{equ:f&g} we can get
\begin{equation*}
4\cdot [j]f=\sum_{k=1}^{2^{v(H)-2}}\frac{[j]h_k}{2^{v(H)-2}}=\sum_{k=1}^{2^{v(H)-3}}\frac{1}{2^{v(H)-2}}\Big([j]h_k+[j]h_{{2^{v(H)-2}-k+1}}\Big)
\end{equation*}
and
\begin{equation*}
4\cdot [j]g=\sum_{k=1}^{2^{v(H)-2}}\frac{[j]h_{2^{v(H)-2}+k}}{2^{v(H)-2}}=\sum_{k=1}^{2^{v(H)-3}}\frac{1}{2^{v(H)-2}}\Big([j]h_{2^{v(H)-2}+k}+[j]h_{2^{v(H)-1}-k+1}\Big).
\end{equation*}
By the above formulas of $[j]h_k$ (note that they all have the same parity), this implies that $[j]f\cdot [j]g\geq 0$. Moreover, we see that $\big|[j]f\big|-\big|[j]g\big|$ equals
\begin{equation*}
\frac{1}{2^{v(H)}}\sum_{k=1}^{2^{v(H)-3}}\Big(\left(\big|[j]h_k\big|+\big|[j]h_{{2^{v(H)-2}-k+1}}\big|\right)-\left(\big|[j]h_{2^{v(H)-2}+k}\big|+\big|[j]h_{2^{v(H)-1}-k+1}\big|\right)\Big).
\end{equation*}

If we let $j=1$, then for any $k$, the $k$-th term in the above formula vanishes as follows
\begin{equation*}
\begin{aligned}
&\frac12\Big(\left(\big|[1]h_k\big|+\big|[1]h_{{2^{v(H)-2}-k+1}}\big|\right)-\left(\big|[1]h_{2^{v(H)-2}+k}\big|+\big|[1]h_{2^{v(H)-1}-k+1}\big|\right)\Big)\\
=&(t_k+\alpha^++2c^++\beta^+)+(t_k+\alpha^-+2c^-+\beta^-)\\
&-(t_k+\alpha^++c^++c^-+\beta^-)-(t_k+\alpha^-+c^-+c^++\beta^+)\\
=&0.
\end{aligned}
\end{equation*}
Since $[j]f\cdot [j]g\geq 0$ holds for any $j$, this shows that $[1]f=[1]g$.

It remains to consider the case when $j=2$.
In this case, for any $k$ we have
\begin{equation*}
\begin{aligned}
&\frac{1}{4}\Big(\left(\big|[2]h_k\big|+\big|[2]h_{{2^{v(H)-2}-k+1}}\big|\right)-\left(\big|[2]h_{2^{v(H)-2}+k}\big|+\big|[2]h_{2^{v(H)-1}-k+1}\big|\right)\Big)\\
=&\binom{t_k+\alpha^++2c^++\beta^+}{2}+\binom{t_k+\alpha^-+2c^-+\beta^-}{2}\\
&-\binom{t_k+\alpha^++c^++c^-+\beta^-}{2}-\binom{t_k+\alpha^-+c^-+c^++\beta^+}{2}\\
=&(c^+-c^-+\alpha^+-\alpha^-)(c^+-c^-+\beta^+-\beta^-)=(2c^++2\alpha^+-A)(2c^++2\beta^+-B).
\end{aligned}
\end{equation*}
Note that this only depends on the values of $\alpha^+,c^+$ and $\beta^+$.
Let $K$ be the number of vertices which are not adjacent to any of $a,b$.
We also have $0\leq \alpha^+\leq A-C, 0\leq c^+\leq C$ and $0\leq \beta^+ \leq B-C$,
Summing over all $k$, by double counting we have that
\begin{equation*}
\begin{aligned}
&\big|[2]f\big|-\big|[2]g\big|\\
=&\frac{1}{2^{v(H)}}\sum_{k=1}^{2^{v(H)-3}}\Big(\left(\big|[2]h_k\big|+\big|[2]h_{{2^{v(H)-2}-k+1}}\big|\right)-\left(\big|[2]h_{2^{v(H)-2}+k}\big|+\big|[2]h_{2^{v(H)-1}-k+1}\big|\right)\Big)\\
=&\frac{2^K}{2^{v(H)+1}}\sum_{c^+=0}^{C}\binom{C}{c^+}\sum_{\alpha^+=0}^{A-C}\binom{A-C}{\alpha^+}\sum_{\beta^+=0}^{B-C}\binom{B-C}{\beta^+}\cdot 4(2c^++2\alpha^+-A)(2c^++2\beta^+-B).
\end{aligned}
\end{equation*}
Therefore, we can further get that $\big|[2]f\big|-\big|[2]g\big|$ is equal to
\begin{equation*}
\begin{aligned}
&\frac{2^K}{2^{v(H)-1}}\sum_{c^+=0}^{C}\binom{C}{c^+}\left(\sum_{\alpha^+=0}^{A-C}\binom{A-C}{\alpha^+}(2c^++2\alpha^+-A)\right)\cdot \left(\sum_{\beta^+=0}^{B-C}\binom{B-C}{\beta^+}(2c^++2\beta^+-B)\right)\\
=&\frac{2^K}{2^{v(H)-1}}\sum_{c^+=0}^{C}\binom{C}{c^+}\left(\frac{1}{2}\sum_{\alpha^+=0}^{A-C}\binom{A-C}{\alpha^+}\big(2c^++2\alpha^+-A+2c^++2(A-C-\alpha^+)-A\big)\right)\\
&\ \ \ \cdot \left(\frac{1}{2}\sum_{\beta^+=0}^{B-C}\binom{B-C}{\beta^+}\big(2c^++2\beta^+-B+2c^++2(B-C-\beta^+)-B\big)\right)\\
=&\frac{2^K}{2^{v(H)-1}}\sum_{c^+=0}^{C}\binom{C}{c^+}\sum_{\alpha^+=0}^{A-C}\binom{A-C}{\alpha^+}\sum_{\beta^+=0}^{B-C}\binom{B-C}{\beta^+}(2c^+-C)^2\\
\geq & 0.
\end{aligned}
\end{equation*}
This proves that $\big|[2]f\big|\geq\big|[2]g\big|$, completing the proof.
\end{proof}

It is easy to see that the proof of this lemma also yields the following corollary.
\begin{cor}\label{coro}
Let $H$ be a graph and $a,b$ be two vertices with $ab\notin E(H)$.
Then $[2]f_{a,b,H}=[2]g_{a,b,H}$ if and only if $N_H(a)\cap N_H(b)=\emptyset$.
\end{cor}

Let $H$ be a graph. Throughout the rest of the paper, for convenience we define $$\Delta_{H}(p):=t_H(U_p)-t_{K_2}(U_p)^{e(H)}=t_H(U_p)-(p-1)^{e(H)}.$$

\begin{lem}\label{lem:p2}
Let $H$ be a graph and $a,b$ be two vertices with $ab\notin E(H)$. If $p^2|\Delta_{H}(p)$, then $$p^2\big|\big(4f_{a,b,H}(p)-(p-1)^{e(H)}\big).$$
\end{lem}

\begin{proof}
Note that $t_H(U_p)=2f_{a,b,H}(p)+2g_{a,b,H}(p)$.
By Lemma~\ref{lem:key}, $[0]f_{a,b,H}=[0]g_{a,b,H}$ and $[1]f_{a,b,H}=[1]g_{a,b,H}$.
This implies that $p^2|\left(4f_{a,b,H}(p)-t_H(U_p)\right)$.
Since $p^2|\Delta_{H}(p)$ and $\Delta_{H}(p)=t_H(U_p)-(p-1)^{e(H)}$,
it follows easily that $p^2|\left(4f_{a,b,H}(p)-(p-1)^{e(H)}\right)$.
\end{proof}

\begin{lem}\label{lem:Delta(H*)}
Let $H$ be a graph and $a,b$ be two vertices with $ab\notin E(H)$.
Let $H^*$ be obtained from $H$ by adding the new edge $ab$. Then
$$t_{H^*}(U_p)=4p\cdot f_{a,b,H}(p)-t_{H}(U_p) \mbox{ ~and~ } \Delta_{H^*}(p)=p\left(4f_{a,b,H}(p)-(p-1)^{e(H)}\right)-\Delta_{H}(p).$$
\end{lem}

\begin{proof}
By definition, we have $t_{H^*}(U_p)=(2p-1)\big(f_{a,b,H}^1(p)+f_{a,b,H}^2(p)\big)-\big(g_{a,b,H}^{12}(p)+g_{a,b,H}^{21}(p)\big)=2(2p-1)f_{a,b,H}(p)-2g_{a,b,H}(p)=4p\cdot f_{a,b,H}(p)-t_{H}(U_p).$
Then it holds that
\begin{equation*}
\begin{aligned}
\Delta_{H^*}(p)&=t_{H^*}(U_p)-(p-1)^{e(H)+1}=4p\cdot f_{a,b,H}(p)-t_H(U_p)-p(p-1)^{e(H)}+(p-1)^{e(H)}\\
&=p\left(4f_{a,b,H}(p)-(p-1)^{e(H)}\right)-\Delta_{H}(p),
\end{aligned}
\end{equation*}
as desired.
\end{proof}

Using the above lemmas, we can show that $\Delta_H(p)$ is always divisible by $p^3$.

\begin{lem}\label{lem:p^3|Delta}
Let $H$ be any graph. Then $p^3|\Delta_H(p)$.
\end{lem}

\begin{proof}
Fix the number $n$ of vertices in $H$.
We prove this lemma by induction on the number of edges of $H$.
If $H$ is an $n$-vertex empty graph, the conclusion holds trivially.
Assume it holds for all $n$-vertex graphs with $k$ edges.
Now consider an $n$-vertex graph $H$ with $k+1$ edges.
Let $ab\in E(H)$ and $H^-$ be obtained from $H$ by deleting $ab$.
By Lemma~\ref{lem:Delta(H*)}, $\Delta_H(p)=p\left(4f_{a,b,H^-}(p)-(p-1)^k\right)-\Delta_{H^-}(p)$.
Using induction on $H^-$, we have $p^3|\Delta_{H^-}(p)$ and then by Lemma~\ref{lem:p2}, we get that $p^3|p\left(4f_{a,b,H^-}(p)-(p-1)^k\right)$.
Thus we obtain that $p^3|\Delta_{H}(p)$, completing the proof.
\end{proof}

Finally, we need the following lemma to indicate the sign of some coefficients in $\Delta_H(p)$.

\begin{lem}\label{lem:sign-p^2}
Let $H$ be a graph and $a,b$ be two vertices with $ab\notin E(H)$.
Then the coefficient of $p^2$ in $4f_{a,b,H}(p)-(p-1)^{e(H)}$ times $(-1)^{e(H)}$ is non-negative.
\end{lem}

\begin{proof}
Let $f=f_{a,b,H}$ and $g=g_{a,b,H}$.
Then Lemma~\ref{lem:p^3|Delta} states $p^3|\left(2f(p)+2g(p)-(p-1)^{e(H)}\right)$.
So $2\cdot [2]f+2\cdot [2]g=\binom{e(H)}{2}(-1)^{e(H)}$.
From Lemma~\ref{lem:key} we see $[2]f\cdot [2]g\geq 0$ and $|[2]f|\geq |[2]g|$.
Thus if $e(H)$ is odd, then $[2]f<0$, $[2]g\leq 0$, and $[2]f\leq [2]g$.
This shows that $4\cdot [2]f-\binom{e(H)}{2}(-1)^{e(H)}\leq 2\cdot [2]f+2\cdot [2]g-\binom{e(H)}{2}(-1)^{e(H)}=0.$
So the conclusion holds.
If $e(H)$ is even, we can derive $[2]f>0$ and $[2]f\geq [2]g$, which imply the same conclusion.
\end{proof}

\section{Proof of Theorem~\ref{thm:main}}

In this section, we present the proof of  Theorem~\ref{thm:main}. We begin with the following two lemmas.

\begin{lem}\label{lem:no-K3}
Let $H$ be a graph which contains no triangles. Then $p^4|\Delta_H(p)$.
\end{lem}

\begin{proof}
Let $(H_0, H_1,...,H_t)$ be a sequence of graphs on the same vertex set $V(H)$ such that $E(H_0)=\emptyset$, $H_{i}$ is obtained from $H_{i-1}$ by adding an edge for each $i\in [t]$, and $H_t=H$.
We will prove $p^4 |\Delta_{H_i}(p)$ using induction on $i$.
When $i=0$, this holds trivially as $\Delta_{H_0}(p)=0$.
Assume the result is true for $i$.
Let $H_{i+1}$ be obtained from $H_i$ by adding a new edge $ab$.
As $H$ has no triangle, we must have that $N_{H_i}(a)\cap N_{H_i}(b)=\emptyset$.
By Lemma~\ref{lem:key} and Corollary~\ref{coro}, we get that $[j]f_{a,b,H_{i}}=[j]g_{a,b,H_{i}}$ for each $0\leq j\leq 2$.
By induction, we have $p^4|\Delta_{H_i}(p)$, where $\Delta_{H_i}(p)=2f_{a,b,H_{i}}+2g_{a,b,H_{i}}-(p-1)^{e(H_i)}$.
Combining the above facts, we can conclude that $p^3|\big(4f_{a,b,H_i}-(p-1)^{e(H_i)}\big)$.
By Lemma~\ref{lem:Delta(H*)}, $\Delta_{H_{i+1}}(p)=p\left(4f_{a,b,H_i}(p)-(p-1)^{e(H_i)}\right)-\Delta_{H_i}(p)$.
Using $p^4|\Delta_{H_i}(p)$ again, we know that $p^4|\Delta_{H_{i+1}}$. This proves the lemma.
\end{proof}

\begin{lem}\label{lem:K3}
Let $H$ be a graph containing some triangle.
If $e(H)$ is odd, then the coefficient of $p^3$ in $\Delta_H(p)$ is positive; otherwise, the coefficient of $p^3$ in $\Delta_{H}(p)$ is negative.
\end{lem}

\begin{proof}
Let $(H_0, H_1,...,H_s)$ be a sequence of graphs on the same vertex set $V(H)$ such that
$H_0$ is a maximal subgraph of $H$ which does not contain triangles, $H_{i}$ is obtained from $H_{i-1}$ by adding an edge for each $i\in [s]$, and $H_s=H$ for some $s\geq 1$.
By Lemma~\ref{lem:no-K3}, since $H_0$ does not contain triangles, it follows that $p^4|\Delta_{H_0}(p)$ and thus $[j]\Delta_{H_0}=0$ for $0\leq j\leq 3$.

To finish the proof, it suffices to show that for each $i\in [s]$, the coefficient $[3]\Delta_{H_i}$ has the same parity as $(-1)^{e(H_i)-1}$.
We first prove for $i=1$.
Let $ab\in E(H_1)\backslash E(H_0)$.
By Lemma~\ref{lem:sign-p^2}, the coefficient of $p^2$ in $4f_{a,b,H_0}-(p-1)^{e(H_0)}$ either is 0 or has the same parity as $(-1)^{e(H_0)}$.
Suppose for a contradiction that this coefficient is 0. So $[2]f_{a,b,H_0}=(-1)^{e(H_0)}\binom{e(H_0)}{2}/4$.
Note that $\Delta_{H_0}(p)=2f_{a,b,H_0}(p)+2g_{a,b,H_0}(p)-(p-1)^{e(H_0)}$ whose coefficient of $p^2$ is also 0.
This implies $[2]f_{a,b,H_0}=(-1)^{e(H_0)}\binom{e(H_0)}{2}/4=[2]g_{a,b,H_0}$.
However by the choice of $H_0$, we see $N_{H_0}(a)\cap N_{H_0}(b)\neq \emptyset$ and thus by Corollary~\ref{coro}, $|[2]f_{a,b,H_0}|>|[2]g_{a,b,H_0}|$, a contradiction.
Therefore, the coefficient of $p^2$ in $4f_{a,b,H_0}(p)-(p-1)^{e(H_0)}$ has the same parity as $(-1)^{e(H_0)}$.
Using Lemma~\ref{lem:Delta(H*)} that $\Delta_{H_{1}}(p)=p\left(4f_{a,b,H_0}(p)-(p-1)^{e(H_0)}\right)-\Delta_{H_0}(p)$ and the fact that $[3]\Delta_{H_0}=0$,
we derive that indeed $[3]\Delta_{H_1}$ has the same parity as $(-1)^{e(H_0)}=(-1)^{e(H_1)-1}$.

Now we may assume that $[3]\Delta_{H_i}$ has the same parity as $(-1)^{e(H_i)-1}$ for some $1\leq i<s$.
We consider $[3]\Delta_{H_{i+1}}$.
Using Lemma~\ref{lem:Delta(H*)}, we have $\Delta_{H_{i+1}}(p)=p\left(4f_{a,b,H_i}(p)-(p-1)^{e(H_i)}\right)-\Delta_{H_i}(p)$.
By Lemma~\ref{lem:sign-p^2}, the coefficient of $p^2$ in $4f_{a,b,H_i}(p)-(p-1)^{e(H_i)}$ either is 0 or has the same parity as $(-1)^{e(H_i)}$.
In either case, we see that $[3]\Delta_{H_{i+1}}$ should have the same parity as $(-1)^{e(H_i)}$, which is $(-1)^{e(H_{i+1})-1}$.
Inductively, this completes the proof.
\end{proof}

Now we are ready to prove Theorem~\ref{thm:main}.

\begin{proof}[\bf Proof of Theorem~\ref{thm:main}.]
Let $H$ be a graph containing some triangle with $e(H)\geq 4$.
By Proposition~\ref{prop}, we want to find a kernel $U_p$ such that $$\sum_{F\in \mathcal{E}^+(H)} \Delta_F(p)=\sum_{F\in \mathcal{E}^+(H)}\left(t_F(U_p)-t_{K_2}(U_p)^{e(F)}\right)<0.$$
Consider any graph $F\in \mathcal{E}^+(H)$.
By Lemma~\ref{lem:p^3|Delta}, we have $p^3|\Delta_F(p)$, and by Lemmas~\ref{lem:no-K3} and \ref{lem:K3},
since $e(F)$ is even, we always can get that $[3]\Delta_F\leq 0.$
Since $e(H)\geq 4$, there must exist some $F\in \mathcal{E}^+(H)$ which contains a triangle.
Thus by Lemma~\ref{lem:K3}, $[3]\Delta_F$ is strictly negative.
Putting the above all together, $p^3|\sum_{F\in \mathcal{E}^+(H)} \Delta_F(p)$, and the coefficient of $p^3$ in $\sum_{F\in \mathcal{E}^+(H)} \Delta_F(p)$ is strictly negative.
Therefore, we can take $p>0$ to be small enough such that $\sum_{F\in \mathcal{E}^+(H)} \Delta_F(p)<0$ holds, finishing the proof.
\end{proof}

\section{Concluding remarks}
In this paper, we prove that any graph properly containing a triangle is not strongly common.
Towards the characterization of strongly common graphs, certain patterns (i.e., the coefficients in $\Delta_H(p)$) appearing in our proof lead us to the following question.

\begin{quest}
Let $k\geq 1$ and $H$ be a graph of girth $2k+1$ with more than $2k+1$ edges. Is it true that $H$ is not strongly common?
\end{quest}

If true, then the recent result of Kim and Lee \cite{KL22} implies that the odd cycle $C_{2k+1}$ is the only connected graph of girth $2k+1$ which is strongly common.
We believe Lemma~\ref{lem:key} (and some similar arguments in the proof) might be useful for this question.

Lastly, we would like to remark that Theorem~\ref{thm:main} can be strengthened as the following: Any graph properly containing a triangle is not locally strongly common.
A graph $H$ is {\it locally strongly common} if for every kernel $U$, there exists $\varepsilon_0>0$ such that $$t_{H}\left(\frac{1}{2}+\varepsilon U\right)+t_{H}\left(\frac{1}{2}-\varepsilon U\right)\geq t_{K_2}\left(\frac{1}{2}+\varepsilon U\right)^{e(H)}+t_{K_2}\left(\frac{1}{2}-\varepsilon U\right)^{e(H)}$$ holds for every $0<\varepsilon<\varepsilon_0.$
Clearly if a graph $H$ is strongly common, then it is locally strongly common.
Using the same expansion as in Section 2, this strengthened statement is equivalent to prove that for any $\varepsilon>0$, there exists a kernel $U$ satisfying that
\begin{equation*}
\sum_{F\in \mathcal{E}^+(H)}\left(t_F(\varepsilon U)-t_{K_2}(\varepsilon U)^{e(F)}\right)=\sum_{F\in \mathcal{E}^+(H)}\varepsilon^{e(F)}\left(t_F(U)-t_{K_2}(U)^{e(F)}\right)<0.
\end{equation*}
It is not hard to see that the proof of Theorem~\ref{thm:main} (straightforwardly) shows that such a kernel $U$ can be chosen to be the kernel $U_p$ for some sufficiently small $p>0$.

\medskip

{\it E-mail address:} mathsch@mail.ustc.edu.cn

\medskip
	
{\it E-mail address:} jiema@ustc.edu.cn

\end{document}